\documentclass[a4paper]{amsart}
\usepackage{graphicx}
\usepackage{amssymb}
\usepackage{amsmath}
\usepackage{amsthm}
\usepackage{amscd}
\usepackage[all,2cell]{xy}

\usepackage[pagebackref,colorlinks]{hyperref}

\UseAllTwocells \SilentMatrices
\newtheorem{thm}{Theorem}[section]

\newtheorem{cor}[thm]{Corollary}

\newtheorem{lem}[thm]{Lemma}
\newtheorem{exm}[thm]{Example}

\newtheorem{prop}[thm]{Proposition}

\theoremstyle{definition}
\newtheorem{defn}[thm]{Definition}
\theoremstyle{remark}
\newtheorem{rem}[thm]{\bf Remark}
\numberwithin{equation}{section}

\begin{document}
\title[Comparing classical tilting modules and $\tau$-tilting modules]{Comparing classical tilting modules and $\tau$-tilting modules}
\author[Chen, Li, Zhang, Zhao] {Xiao-Wu Chen, Zhi-Wei Li, Xiaojin Zhang$^*$, Zhibing Zhao}

\makeatletter
\@namedef{subjclassname@2020}{\textup{2020} Mathematics Subject Classification}
\makeatother

\thanks{$^*$ The corresponding author}
 \date{the revised version,  \today}

\thanks{xwchen$\symbol{64}$mail.ustc.edu.cn, zhiweili@jsnu.edu.cn, xjzhang@jsnu.edu.cn, zbzhao@ahu.edu.cn}

\maketitle

\dedicatory{}%
\commby{}%

\begin{abstract}
We characterize support $\tau$-tilting modules as classical tilting modules over quotient algebras satisfying a tensor-vanishing condition, and characterize classical tilting modules as $\tau$-tilting modules satisfying a ${\rm Tor}^1$-vanishing condition. We use delooping levels to study  \emph{Self-orthogonal $\tau$-tilting Conjecture}: any self-orthogonal $\tau$-tilting module is classical tilting. We confirm the conjecture when the endomorphism algebra of the module has finite global delooping level.
\end{abstract}

\section{Introduction}
 The $\tau$-tilting theory \cite{AIR} is a natural generalization of the classical tilting theory \cite{HR, Bon}. Related ideas might be traced back to \cite{AS}. It is closely related to the silting theory \cite{AI} and cluster tilting theory \cite{IY}; see also \cite{DF}.

The central objects in $\tau$-tilting theory are support $\tau$-tilting modules. One advantage is that they admit a well-defined mutation operation. There are well-known bijections between basic support $\tau$-tilting modules and several important objects in representation theory, including functorially finite torsion classes, basic two-term silting complexes \cite{AIR}, and intermediate algebraic $t$-structures \cite{KY}.

Throughout the paper, by a tilting module, we always mean a classical tilting module. Let $A$ be an artin algebra. It is well known that a support $\tau$-tilting $A$-module $T$ becomes a tilting module over the quotient algebra $A/{{\rm Ann}(T)}$. Here, ${\rm Ann}(T)$ is the annihilator of $T$. However, the converse is not true in general. We are interested in the following basic question: to what extent, a tilting module over a quotient algebra of $A$ becomes a support $\tau$-tilting $A$-module?

Recall that an $A$-module $T$ is called rigid if ${\rm Ext}^1_A(T, T)=0$, and called self-orthogonal if ${\rm Ext}_A^i(T, T)=0$ for any $i\geq 1$. An $A$-module $T$ is called \emph{tilting} if it is rigid and satisfies ${\rm pd}_A(T)\leq 1$ and $|T|=|A|$. Here, ${\rm pd}_A (T)$ denotes the projective dimension of $T$ and $|T|$ denotes the number of isomorphism classes of indecomposable direct summands of $T$. 

An $A$-module $T$ is called $\tau$-rigid if ${\rm Hom}_A(T, \tau(T))=0$, where $\tau$ denotes the Auslander-Reiten translation. A $\tau$-rigid $A$-module $T$ is called \emph{$\tau$-tilting} if it satisfies $|T|=|A|$. An $A$-module $T$ is called \emph{support $\tau$-tilting} if there is an idempotent ideal $I$ such that $IT=0$ and that $T$ is a $\tau$-tilting $A/I$-module.

The first result answers the question above, using a new tensor-vanishing condition; see Theorem~\ref{thm:supp}.

\vskip 5pt

\noindent {\bf Theorem~I}.\;  \emph{Let $T$ be an $A$-module. Set $I={\rm Ann}(T)$ and $\bar{A}=A/I$. Then the $A$-module $T$ is support $\tau$-tilting if and only if the corresponding $\bar{A}$-module $T$ is tilting satisfying $I\otimes_A T=0$.}

\vskip 5pt
In contrast to the trivial fact ${\rm Ann}(T)T=0$, the condition ${\rm Ann}(T)\otimes_A T=0$ above is nontrivial and  necessary.

By \cite{ASS, AIR}, a tilting module is precisely a faithful $\tau$-tilting module.  This might be viewed as a characterization of a tilting module in terms of a $\tau$-tilting module. The second result is another such characterization, using a new ${\rm Tor}^1$-vanishing condition; see Theorem~\ref{thm:tilting}.

\vskip 5pt

\noindent {\bf Theorem~II}. \; \emph{Let $T$ be an $A$-module.  Set $I={\rm Ann}(T)$. Then $T$ is tilting if and only if it is faithful $\tau$-tilting, if and only if it is $\tau$-tilting satisfying ${\rm Tor}^A_1(I, T)=0$.}
\vskip 5pt

The first ``if and only if" follows from \cite[Proposition 2.2(b)]{AIR}. The second ``only if" part of Theorem~II is trivial, since the annihilator of  any  tilting module is  zero. We emphasize that the proofs of Theorems~I and II are quite elementary.

 It is clear that any tilting module is a self-orthogonal $\tau$-tilting module. The  following natural question is asked in \cite{Zhang}: is any self-orthogonal $\tau$-tilting module  tilting? Since its positive answer is implied by a conjecture in \cite{Eno}, we propose to call it the {\em Self-orthogonal $\tau$-tilting Conjecture},  and (S$\tau$C) for short.

We mention that (S$\tau$C) holds for algebras with finite global dimension. Indeed, it is shown in \cite{Zhang} that any self-orthogonal $\tau$-tilting module with finite projective dimension is tilting.  More generally, by \cite{Mar}, it holds for syzygy-finite Gorenstein algebras, since any self-orthogonal module over such algebras has finite projective dimension; compare \cite{Chang, LyW}.  For more confirmed cases, we refer to Proposition~\ref{prop:conj}. We mention recent progress on (S$\tau$C) in  \cite{Zhang2, ZhangZhou}.

We use the delooping level in \cite{Gel} to investigate (S$\tau$C). Recall that the \emph{global delooping level} of an algebra is the supremum of the delooping levels of all its left modules.  Examples of algebras  with finite global delooping level include Gorenstein algebras and syzygy-finite algebras.

We mention that delooping levels play a role in the finitistic dimension conjecture. The following result indicates that they might be useful to study (S$\tau$C); see Theorem~\ref{thm:dell-B}.
\vskip 5pt

\noindent {\bf Theorem~III}.\;  \emph{Let $T$ be a self-orthogonal $\tau$-tilting $A$-module. Set $B={\rm End}_A(T)^{\rm op}$. Then $T$ is tilting provided that $B$ has finite global delooping level.}

\vskip 5pt

 The paper is structured as follows. In Section~2, we recall basic facts in tilting theory. We compare tilting modules and $\tau$-tilting modules in Section~3. We study the delooping level in an exact category in Section~4. In the final section, we discuss homological conjectures on self-orthogonal modules, and prove  Theorem~III, whose proof relies on Theorem~\ref{thm:tilting}.

\section{Preliminaries}

In this section, we recall from \cite{ASS, AIR} basic facts on tilting modules and $\tau$-tilting modules.

 Let $A$ be an artin algebra. Denote by $A\mbox{-mod}$ the abelian category of finitely generated left $A$-modules.  For each $A$-module $M$, we denote by $|M|$ the number of isomorphism classes of indecomposable direct summands of $M$. For example, $|A|$ is the number of isomorphism classes of indecomposable projective $A$-modules, which equals the number of isomorphism classes of simple $A$-modules.

 For a projective $A$-module $P$, its \emph{trace ideal} \cite[\S 8]{AF} is defined to be
 $${\rm tr}(P)=\sum_{f\in {\rm Hom}_A(P, A)}\; {\rm Im}(f).$$
 If $P\simeq A e$ for some idempotent $e$, then ${\rm tr}(P)=AeA$. We observe
 \begin{align}\label{equ:tr}
 {\rm tr} (P) P=P.
 \end{align}
 Recall that a two-sided ideal $I$ of $A$ is \emph{idempotent} if $I=I^2$. It is well known that an ideal $I$ is idempotent if and only if there is a projective module $P$ satisfying $I={\rm tr}(P)$; see \cite[page 667]{APT}.

Let $I$ be an ideal of $A$. Denote by $I_0\subseteq I$ its \emph{stable part}, that is, the largest idempotent ideal contained in $I$. We mention that $I^N=I_0$ for any sufficiently large $N$.  Up to isomorphism, there is a unique basic projective $A$-module $P_0$ with $I_0={\rm tr}(P_0)$. Set
$${\rm st}(I)=|P_0|,$$
which might be called the \emph{stable index} of $I$.

\begin{lem}\label{lem:card}
Let $I$ be an ideal of $A$. Then we have
$$|A|=|A/I|+{\rm st}(I)=|A/I_0|+{\rm st}(I).$$
Moreover, $|A|=|A/I|$ if and only if $I$ is nilpotent.
\end{lem}

\begin{proof}
   For the two equalities,  we apply the following facts: for a projective $A$-module $P$, $(A/I)\otimes_A P=0$ if and only if $IP=P$,  which is equivalent to ${\rm tr}(P)\subseteq I$ by (\ref{equ:tr}) ;  moreover, ${\rm st}(I)={\rm st}(I_0)$. For the last statement, we observe that ${\rm st}(I_0)=0$ if and only if $I_0=0$, which is equivalent to the condition that the ideal $I$ is nilpotent.
\end{proof}

Let $T$ be an $A$-module. Denote by ${\rm add}(T)$ the full subcategory of $A\mbox{-mod}$ formed by direct summands of finite direct sums of $T$, and  by ${\rm fac}(T)$ the full subcategory formed by factor modules of finite direct sums of $T$.

Recall that an \emph{exact category} in the sense of \cite{Qui73} is an additive category with a chosen classes of short exact sequences, called conflations, which satisfy certain axioms; see \cite[Appendix~A]{Kel90}. For example, an extension-closed full subcategory of an abelian category inherits the short exact sequences and becomes an exact category.

The following fact about $\tau$-rigid modules is standard.

\begin{lem}\label{lem:par-tilting}
    Let $T$ be an $A$-module. Then $T$ is a $\tau$-rigid $A$-module if and only if ${\rm Ext}_A^1(T, {\rm fac}(T))=0$. In this case, the subcategory ${\rm fac}(T)$ of $A\mbox{-}{\rm mod}$ is closed under extensions, and becomes an exact category. Moreover, $T$ is projective in the exact category ${\rm fac}(T)$.
\end{lem}
\begin{proof} The first statement is \cite[Corollary~5.9]{AS}. The second statement is \cite[Proposition~5.5]{AS}.
\end{proof}

The following well-known  result can be found in \cite[VI.2.5~Theorem~(c) and (d)]{ASS}. We denote by $D$ the Matlis duality.

\begin{lem}\label{lem:1-tilt}
Let $T$ be a  tilting $A$-module. Then the exact category  ${\rm fac}(T)$ has enough projective objects and enough injective objects. Moreover, projective objects are precisely modules in ${\rm add}(T)$, and injective objects are precisely injective $A$-modules.
\end{lem}

\begin{proof}
The proof uses the following fact: for each $M\in {\rm fac}(T)$, take a right ${\rm add}(T)$-approximation $f\colon T^n\rightarrow M$ of $M$. Then $f$ is surjective, whose kernel lies in ${\rm fac}(T)$. For the last statement, we recall that $DA$ belongs to ${\rm fac}(T)$; see   \cite[VI.2.2~Lemma~(d)]{ASS}.
\end{proof}

Recall that an $A$-module $T$ is called \emph{partial tilting} if it is rigid and ${\rm pd}_A (T)\leq 1$. The following result is due to \cite[VIII.5.1~Lemma]{ASS}; compare \cite[Proposition~1.4]{AIR}.

\begin{lem}\label{lem:faithful}
    Let $T$ be an $A$-module. Then the following statements hold.
    \begin{enumerate}
        \item If $T$ is faithful and $\tau$-rigid, then it is partial tilting.
        \item The $A$-module $T$ is faithful and $\tau$-tilting if and only if it is tilting.
    \end{enumerate}
\end{lem}

The following results are well known.

\begin{prop}\label{prop:numerical}
    Let $T$ be an $A$-module with $I={\rm Ann}(T)$. Then the following statements hold.
        \begin{enumerate}
        \item If $T$ is  $\tau$-rigid, we have $|T|\leq |A/I|$.
        \item The $A$-module $T$ is support $\tau$-tilting if and only if it is $\tau$-rigid satisfying $|T|=|A/I|$.
        \item The $A$-module $T$ is $\tau$-tilting if and only if it is support $\tau$-tilting and the ideal $I$ is nilpotent.
    \end{enumerate}
\end{prop}

\begin{proof}
For (1), we refer to \cite[Proposition~1.3]{AIR}; compare \cite[VIII.5.3~Lemma]{ASS}. For (2), we consider the stable part $I_0$ of $I$. Then the $A$-module $T$ is support $\tau$-tilting if and only if it is a $\tau$-tilting $A/I_0$-module. In particular, we have $|T|=|A/I_0|=|A/I|$; see Lemma~\ref{lem:card}.

For (3), we just observe that a support $\tau$-tilting $A$-module is $\tau$-tilting if and only if $I_0=0$, which is equivalent to the nilpotency of $I$.
\end{proof}

\section{New comparison results}

In this section, we compare $\tau$-tilting modules and tilting modules. Corollary~\ref{car:tau-tilting} and Theorem \ref{thm:tilting} are new characterizations of $\tau$-tilting modules and tilting modules, respectively.

We first recall the following lemma concerning modules over quotient algebras.

\begin{lem}\label{new-homo} Let $X$ be an $A$-module and let $I$ be an ideal of $A$.
\begin{enumerate}
\item  Set $\bar{A}=A/{\rm Ann}(X)$. Then $D\bar{A}$ belongs to ${\rm fac}(X)$.
\item Set $\bar{A}=A/I$. For any $i\geq 0$, we have the following natural isomorphisms of functors on $\bar{A}$-${\rm mod}$
$${\rm Ext}_A^{i+1}(-,D\bar{A})\simeq D{\rm Tor}_{i+1}^A(\bar{A},-)\simeq D{\rm Tor}_{i}^A(I,-).\ \ \ \ $$
\end{enumerate}
\end{lem}

\begin{proof} (1) follows from the well-known fact that $X$ is a faithful $\bar{A}$-module; compare \cite[VI.2.2~Lemma~(d)]{ASS}.

(2) The first isomorphism is well known. The second one follows from   dimension shifting if $i\geq 1$. We show the case $i=0$. For any $\bar{A}$-module $M$, applying the functor $-\otimes_A M$ to the exact sequence
$$0\longrightarrow I\longrightarrow A\longrightarrow \bar{A}\longrightarrow 0,$$
one gets the following long exact sequence
$$0\longrightarrow{\rm Tor}_1^A(\bar{A},M)\longrightarrow I\otimes_A M\longrightarrow A\otimes_A M\longrightarrow \bar{A}\otimes_A M\longrightarrow 0.$$
Then the assertion follows from the fact $IM=0$ for any $\bar{A}$-module $M$.
\end{proof}

We then characterize $\tau$-rigid modules via certain partial tilting modules over quotient algebras.

\begin{prop}\label{prop:tau-rigid}
    Let $T$ be an $A$-module. Set $I={\rm Ann}(T)$ and $\bar{A}=A/I$. Then the $A$-module $T$ is $\tau$-rigid if and only if the corresponding $\bar{A}$-module $T$ is partial tilting satisfying $I\otimes_A T=0$.
\end{prop}

\begin{proof}
    For the ``only if" part, we assume that the $A$-module $T$ is $\tau$-rigid. Then the corresponding $\bar{A}$-module $T$  is also $\tau$-rigid, which is faithful. By Lemma~\ref{lem:faithful}(1), it is partial tilting.

By Lemma~\ref{new-homo}(1), $D\bar{A}$ belongs to ${\rm fac}(T)$. By Lemma~\ref{lem:par-tilting}, we have ${\rm Ext}^1_A(T, D\bar{A})=0$. By Lemma \ref{new-homo}(2), we have ismorphisms
    $${\rm Ext}^1_A(T, D\bar{A}) \simeq D{\rm Tor}_1^A(\bar{A}, T) \simeq D(I\otimes_A T).$$
  We infer that $I\otimes_A T=0$.

    Conversely, we assume that the corresponding $\bar{A}$-module $T$ is partial tilting satisfying $I\otimes_A T=0$. The above proof yields ${\rm Ext}_A^1(T, D\bar{A})=0$. Consequently, ${\rm Ext}_A^1(T, E)=0$ for any injective $\bar{A}$-module $E$.

    Take any object $X\in {\rm fac}(T)$. We form an exact sequence
    $$0\longrightarrow X\stackrel{a}\longrightarrow E\stackrel{c}\longrightarrow Y\longrightarrow 0$$
    in $\bar{A}\mbox{-mod}$ with $E$ injective. We observe that $Y$ also belongs to ${\rm fac}(T)$, since $D\bar{A}$ belongs to ${\rm fac}(T)$; see Lemma~\ref{new-homo}(1). By Lemma~\ref{lem:par-tilting}, $T$ is projective in ${\rm fac}(T)$.  In particular, ${\rm  Hom}_A(T, c)$ is surjective. Consequently, we infer that $a$ induces an injective map
    $${\rm Ext}_A^1(T, X)\longrightarrow {\rm Ext}_A^1(T, E).$$
    We conclude that ${\rm Ext}_A^1(T, X)=0$. By Lemma \ref{lem:par-tilting}, this implies that $T$ is $\tau$-rigid.
\end{proof}

\begin{rem}
    The vanishing condition $I\otimes_A T=0$ above is necessary. Take any ideal $I$ of $A$. Then $A/I$ is certainly a partial tilting module over $A/I$. Applying Proposition~\ref{prop:tau-rigid} to $A/I$ and using the isomorphism $I\otimes_A A/I\simeq I/I^2$, we infer that the $A$-module $A/I$ is $\tau$-rigid if and only if the ideal $I$ is idempotent.
\end{rem}

\begin{thm}\label{thm:supp}
    Let $T$ be an $A$-module. Set $I={\rm Ann}(T)$ and $\bar{A}=A/I$. Then the $A$-module $T$ is support $\tau$-tilting if and only if the corresponding $\bar{A}$-module $T$ is tilting satisfying $I\otimes_A T=0$.
\end{thm}

\begin{proof}
    For the  ``only if" part, we assume that the $A$-module $T$ is support $\tau$-tilting. By Proposition~\ref{prop:tau-rigid}, one gets that $T$ is a partial tilting $\bar{A}$-module satisfying $I\otimes_A T=0$. By Proposition~\ref{prop:numerical}(2), we have $|T|=|\bar{A}|$, and hence $T$ is a tilting $\bar{A}$-module.

    Conversely, we assume that the corresponding $\bar{A}$-module $T$ is tilting satisfying $I\otimes_A T=0$. It is obvious that $|T|=|\bar{A}|$. By Proposition~\ref{prop:tau-rigid}, one gets that $T$ is a $\tau$-rigid $A$-module. By Proposition~\ref{prop:numerical}(2), $T$ is a support $\tau$-tilting $A$-module.
\end{proof}

\begin{rem}\label{rem:fac}
    Let $T$ be a support $\tau$-tilting $A$-module. The corresponding $\bar{A}$-module $T$ is tilting. By Lemma~\ref{lem:1-tilt}, the exact category ${\rm fac}(T)$ has enough projective objects, and the full subcategory of projective objects coincides with ${\rm add}(T)$; compare \cite[Proposition~2.5]{Zhang}. We emphasize that ${\rm fac}(T)$ inherits the same exact structure both from $A\mbox{-mod}$ and  $\bar{A}\mbox{-mod}$.
\end{rem}

\begin{cor}\label{car:tau-tilting}
    Let $T$ be an $A$-module. Set $I={\rm Ann}(T)$ and $\bar{A}=A/I$. Then the $A$-module $T$ is  $\tau$-tilting if and only if the ideal $I$ is nilpotent and the corresponding $\bar{A}$-module $T$ is tilting satisfying $I\otimes_A T=0$.
\end{cor}

\begin{proof}
    We just combine Theorem~\ref{thm:supp} and Proposition~\ref{prop:numerical}(3).
\end{proof}

In the following result, the equivalence between (1) and (2) is essentially  due to \cite[Corollary~3.12]{LW}, which strengthens \cite[Theorem~3.2]{Zhang}.

\begin{thm}\label{thm:tilting}
    Let $T$ be a $\tau$-tilting $A$-module. Set $I={\rm Ann}(T)$ and $\bar{A}=A/I$. Then the following statements are equivalent.
    \begin{enumerate}
        \item The $A$-module $T$ is tilting.
        \item ${\rm Ext}_A^2(T, {\rm fac}(T))=0$.
        \item ${\rm Ext}_A^2(T, D\bar{A})=0$.
        \item ${\rm Tor}^A_1(I, T)=0$.
    \end{enumerate}
\end{thm}

\begin{proof}
    Since any tilting module has projective dimension at most one. Then we have ``$(1)\Rightarrow (2)$". The implication ``$(2)\Rightarrow (3)$" is trivial, since by Lemma \ref{new-homo}(1) $D\bar{A}$ belongs to ${\rm fac}(T)$. 
    
    By Lemma \ref{new-homo}(2) we have the following well-known isomorphisms
    $${\rm Ext}_A^2(T, D\bar{A}) \simeq D{\rm Tor}_2^A(\bar{A}, T) \simeq D{\rm Tor}_1^A(I, T).$$
    Then we have the equivalence between (3) and (4).

    It remains to prove ``$(3)\Rightarrow (1)$". Since by Corollary~\ref{car:tau-tilting} the $\bar{A}$-module $T$ is tilting, we have an exact sequence
    $$0\longrightarrow \bar{A}\longrightarrow T^0 \longrightarrow T^1\longrightarrow 0$$
   with both $T^i\in {\rm add}(T)$. Applying ${\rm Hom}_A(-, D\bar{A})$ to it, we obtain an exact sequence
   $$ {\rm Ext}_A^1(T^0, D\bar{A}) \longrightarrow  {\rm Ext}_A^1(\bar{A}, D\bar{A})\longrightarrow {\rm Ext}_A^2(T^1, D\bar{A}).$$
   Since $D\bar{A}$ belongs to ${\rm fac}(T)$ and $T$ is projective in ${\rm fac}(T)$,
   by Lemma \ref{lem:par-tilting} we have
   $${\rm Ext}_A^1(T^0, D\bar{A})={\rm Ext}^1_{{\rm fac}(T)}(T^0, D\bar{A})=0.$$
   By the assumption in (3), we have  ${\rm Ext}_A^2(T^1, D\bar{A})=0$.
Consequently, we infer that  ${\rm Ext}_A^1(\bar{A}, D\bar{A})=0$.

We now use the following isomorphisms in Lemma \ref{new-homo}(2)
   $${\rm Ext}_A^1(\bar{A}, D\bar{A})\simeq D{\rm Tor}_1^A(\bar{A}, \bar{A}) \simeq D(I\otimes_A \bar{A}) \simeq D(I/I^2). $$
   We conclude that $I=I^2$. Since $I$ is nilpotent by Corollary~\ref{car:tau-tilting}, we infer that $I=0$, that is, the $A$-module $T$ is faithful. Then we are done by Lemma~\ref{lem:faithful}(2).
\end{proof}

\begin{rem}
    Let $T$ be a $\tau$-tilting $A$-module, which is not tilting. Set $I={\rm Ann}(T)$. The two theorems above imply that $I\otimes_A T=0$ but ${\rm Tor}_1^A(I, T)\neq 0$.
\end{rem}

\section{The delooping level}
In this section, we study delooping levels \cite{Gel} in an exact category. The key observation is that finite delooping levels play a role in obtaining  the ${\rm Ext}^2$-vanishing condition in Theorem~\ref{thm:tilting}; see Proposition~\ref{prop:dell-Ext}.

Let $\mathcal{E}$ be an exact category \cite{Qui73} with enough projective objects. Denote by $\mathcal{P}$ the full subcategory formed by projective objects. The projectively stable category is denoted by $\underline{\mathcal{E}}$. For each object $X$, we take a conflation
$$0\longrightarrow \Omega_\mathcal{E}(X)\longrightarrow P \longrightarrow X\longrightarrow 0$$
with $P$ projective. This gives rise to the syzygy endofunctor $\Omega_\mathcal{E}\colon \underline{\mathcal{E}}\rightarrow \underline{\mathcal{E}}$.

The following notion is a categorical analogue of \cite[Definition~1.2]{Gel}. We mention its derived analogue in \cite{GuoI}.

\begin{defn}
  Let $X$ be an object in $\mathcal{E}$. Its \emph{delooping level}, denoted by ${\rm dell}_\mathcal{E}(X)$, is defined to be minimal nonnegative number $n$ such that $\Omega_\mathcal{E}^n(X)$ is isomorphic to a direct summand of $\Omega_\mathcal{E}^{n+1}(N)$ for some object $N$ in $\underline{\mathcal{E}}$. If there is no such a number $n$, we set ${\rm dell}_\mathcal{E}(X)=+\infty$.
\end{defn}

The \emph{global delooping level} of $\mathcal{E}$, denoted by ${\rm gl.dell}(\mathcal{E})$, is defined to be the supremum of ${\rm dell}_\mathcal{E}(X)$ for all objects $X$.

We observe that ${\rm dell}_\mathcal{E}(X)\leq {\rm pd}_\mathcal{E}(X)$. Here, ${\rm pd}_\mathcal{E}$ denotes the projective dimension in $\mathcal{E}$. Therefore, we have ${\rm gl.dell}(\mathcal{E})\leq {\rm gl.dim}(\mathcal{E})$.

Denote by $\Omega^n_\mathcal{E}(\mathcal{E})$ the full subcategory of $\mathcal{E}$ formed by direct summands of $\Omega_\mathcal{E}^n(X)\oplus P$ for some object $X$ and projective object $P$. Then we have a descending chain of subcategories.
$$\mathcal{E}\supseteq \Omega^1_\mathcal{E}(\mathcal{E}) \supseteq \Omega^2_\mathcal{E}(\mathcal{E})\supseteq \cdots \supseteq \mathcal{P}$$

\begin{exm}\label{exm:dell}
{\rm Assume that the exact category $\mathcal{E}$ is \emph{syzygy-stable}, that is,  there exists $d\geq 0$ such that $\Omega^d_\mathcal{E}(\mathcal{E})=\Omega^{d+1}_\mathcal{E}(\mathcal{E})$. Then ${\rm gl.dell}(\mathcal{E})\leq d$.

A particular case is of interest. Assume that $\mathcal{E}$ is Krull-Schmidt, which is \emph{syzygy-finite}, that is, there exists an object $E$ such that $\Omega^d_\mathcal{E}(\mathcal{E})={\rm add}(E)$ for some $d\geq 0$. Then the descending chain above is stable. It follows that $\mathcal{E}$ is syzygy-stable.}
\end{exm}

Let $A$ be an artin algebra. The delooping level of an $A$-module $X$ is denoted by ${\rm dell}_A(X)$. Write ${\rm gl.dell}(A)={\rm gl.dell}(A\mbox{-mod})$, called the \emph{global delooping level} of the algebra $A$.

\begin{exm}
{\rm (1) Assume that $A$ is $d$-Gorenstein, that is, the selfinjective dimension of $A$ on each side is at most $d$. Then $A\mbox{-mod}$ is syzygy-stable. More precisely, we have $\Omega^d_A(A\mbox{-mod})=\Omega^{d+1}_A(A\mbox{-mod})$. Consequently, we have ${\rm gl.dell}(A)\leq d$.

(2) The algebra $A$ is called syzygy-finite if so is $A\mbox{-mod}$. Syzygy-finite algebras include algebras of finite representation type, torsionless-finite algebras \cite{Rin},  and monomial ideals by \cite[Theorem~I]{Zim}. Therefore, the global delooping levels of these algebras are finite.}
\end{exm}

In what follows, we fix an abelian category $\mathcal{A}$ with enough projectives, which form the full subcategory $\mathcal{W}$. For any $T$ in $\mathcal{A}$, denote by ${\rm pd}_{\mathcal{A}}(T)$ the projective dimension of $T$. Let $\mathcal{E}\subseteq \mathcal{A}$ be a full subcategory which is closed under extensions. We assume further that the exact category $\mathcal{E}$ has enough projective objects, which form the full subcategory $\mathcal{P}$.

\begin{prop}\label{prop:dell-Ext}
 Assume that each object in $\mathcal{P}$ is self-orthogonal in $\mathcal{A}$. Suppose that $T\in \mathcal{P}$ and $X\in \mathcal{E}$. Then we have  ${\rm Ext}_\mathcal{A}^2(T, X)=0$ provided that ${\rm pd}_\mathcal{A}(T)<\infty$ or ${\rm dell}_\mathcal{E}(X)<\infty$.
\end{prop}

We emphasize  that for objects $X, Y\in \mathcal{E}$, the natural map
$${\rm Ext}_\mathcal{E}^2(X, Y)\longrightarrow {\rm Ext}_\mathcal{A}^2(X, Y)$$
is injective, but not surjective in general; see \cite[Proposition~3.4]{DI}. For the proof of Proposition~\ref{prop:dell-Ext}, we need the following easy observations.

\begin{lem}\label{lem:Ext}
    Keep the same assumptions above. For any $T\in \mathcal{P}$ and any object $Y\in \mathcal{E}$, the following statements hold.
    \begin{enumerate}
        \item ${\rm Ext}^i_\mathcal{A}(T, Y)\simeq {\rm Ext}^{i+1}_\mathcal{A}(T, \Omega_\mathcal{E}(Y))$ for $i\geq 1$.
        \item ${\rm Ext}^j_\mathcal{A}(T, \Omega_\mathcal{E}^k(Y))=0$ for $1\leq j\leq k+1$.
    \end{enumerate}
\end{lem}

\begin{proof}
For (1), we consider the conflation
$$0\longrightarrow \Omega_\mathcal{E}(Y)\longrightarrow P \longrightarrow Y\longrightarrow 0$$
with $P$ projective in $\mathcal{E}$. The self-orthogonality of $\mathcal{P}$ implies that ${\rm Ext}_\mathcal{A}^i(T, P)=0$ for any $i\geq 1$. Applying ${\rm Hom}_\mathcal{A}(T, -)$ to the conflation above, we infer (1).

Since $\mathcal{E}$ is closed under extensions in $\mathcal{A}$, we have ${\rm Ext}^1_\mathcal{A}(T, Y)={\rm Ext}^1_\mathcal{E}(T, Y)=0$. This yields (2) in the case $j=1$. By induction, the general case follows immediately from (1).
\end{proof}

\noindent \emph{Proof of Proposition~\ref{prop:dell-Ext}.}\;
By applying  Lemma~\ref{lem:Ext}(1) repeatedly,  we have an isomorphism
\begin{align}\label{equ:Ext}
    {\rm Ext}^{2}_\mathcal{A}(T, X)\simeq {\rm Ext}^{n+2}_\mathcal{A}(T, \Omega_\mathcal{E}^n(X)).
\end{align}
for any $n\geq 0$. If ${\rm pd}_\mathcal{A}(T)$ is finite, we are done.

We assume that ${\rm dell}_\mathcal{E}(X)=d$. There exist some object $X'\in \mathcal{E}$ and an object $P\in \mathcal{P}$ such that $\Omega_\mathcal{E}^d(X)$ is isomorphic to a direct summand of $\Omega_\mathcal{E}^{d+1}(X')\oplus P$. By Lemma~\ref{lem:Ext}(2), we have
$${\rm Ext}^{d+2}_\mathcal{A}(T, \Omega_\mathcal{E}^{d+1}(X')\oplus P)=0.$$
Consequently, we have
$${\rm Ext}^{d+2}_\mathcal{A}(T, \Omega_\mathcal{E}^d(X))=0.$$
Combining this with (\ref{equ:Ext}), we infer the required vanishing. \hfill $\square$

\begin{rem}
    The proof above yields a slightly stronger result. Fix $d\geq 0$.  Assume that ${\rm Ext}_\mathcal{A}^i(P, P)=0$ for any $P\in \mathcal{P}$ and $2\leq i\leq d+2$. Then we still have
    ${\rm Ext}^{2}_\mathcal{A}(T, X)=0$, provided that ${\rm pd}_\mathcal{A}(T)\leq d+2$ or  ${\rm dell}_\mathcal{E}(X)\leq d$.
\end{rem}

\section{Homological conjectures}

In this section, we study the self-orthogonal $\tau$-tilting conjecture \cite{Zhang}, which is implied by  a conjecture in \cite{Eno}. Theorem~\ref{thm:dell-B} confirms the conjecture under the assumption that the endomorphism algebra of the $\tau$-tilting module has finite global delooping level.

In what follows, we assume that $A$ is an artin algebra. The following conjecture is posted in \cite[Conjecture~5.9]{Eno}.

\vskip 5pt

\noindent \emph{Self-orthogonal Wakamatsu-tilting Conjecture} (SWC). \; Let $T$ be a self-orthogonal $A$-module with $|T|=|A|$. Then $T$ is Wakamatsu-tilting.

\vskip 5pt

We recall from \cite{Wak, MR} that a self-orthogonal $A$-module $W$ is called \emph{Wakamatsu-tilting} if there is a long exact sequence
$$0\longrightarrow A\longrightarrow T^0\longrightarrow T^1\longrightarrow T^2\longrightarrow \cdots$$
with each $T^i\in {\rm add}(W)$ and each cocycle in $^{\perp} W=\{M\; |\; {\rm Ext}_A^i(M, W)=0, i\geq 1\}$. We observe that a Wakamatsu-tilting module is necessarily faithful.

\vskip 5pt

\noindent \emph{Self-orthogonal Faithful Conjecture} (SFC). \; Let $T$ be a self-orthogonal $A$-module with $|T|=|A|$. Then $T$ is faithful.

\vskip 5pt

We mention that if $A$ is selfinjective, then (SFC) is equivalent to the following well-known Tachikawa's Conjecture \cite{Tach}: any self-orthogonal module over a selfinjective algebra is projective. Here, we use the well-known fact on a selfinjective algebra $A$: any faithful $A$-module $T$ with $|T|=|A|$ is necessarily a projective generator.

Recall that any $\tau$-tilting $A$-module $T$ satisfies $|T|=|A|$. Since any faithful $\tau$-tilting module is tilting, we infer that (SFC) implies the following conjecture posted in \cite{Zhang}.

\vskip 5pt

\noindent \emph{Self-orthogonal $\tau$-tilting Conjecture} (S$\tau$C). \; Let $T$ be a self-orthogonal $\tau$-tilting $A$-module. Then $T$ is tilting.

\vskip 5pt

In summary, we have the following implications for any given algebra $A$.
\begin{center}
     (SWC) $\Rightarrow$ (SFC) $\Rightarrow$  (S$\tau$C)
\end{center}

\begin{rem}\label{rem:local}
    It is well known that for a local algebra $A$,  the only basic $\tau$-tilting $A$-module is isomorphic to $A$; see \cite[Example~6.1]{AIR}. Therefore, (S$\tau$C) holds trivially for local algebras.
\end{rem}

Let $T$ be a $\tau$-tilting $A$-module. Then the exact category ${\rm fac}(T)$ has enough projective objects, which are precisely modules in ${\rm add}(T)$; see Remark~\ref{rem:fac}. The delooping levels in the exact category ${\rm fac}(T)$ will be denoted by ${\rm dell}_T$. It might be viewed as a relative version of the delooping level in \cite{Gel}.

The following result is partly due to \cite[Theorem~1.3]{Zhang}.

\begin{prop}\label{prop:dell}
    Let $T$ be a self-orthogonal $\tau$-tilting $A$-module. Set $\bar{A}=A/{{\rm Ann}(T)}$.  Assume that either ${\rm pd}_A(T)<+\infty$ or ${\rm dell}_T(D\bar{A})<+\infty$. Then  $T$ is tilting.
\end{prop}

\begin{proof}
    By Remark~\ref{rem:fac}, the assumptions in Proposition~\ref{prop:dell-Ext} hold for ${\rm fac}(T)\subseteq A\mbox{-mod}$. Applying Proposition~\ref{prop:dell-Ext}, we obtain ${\rm Ext}_A^2(T, D\bar{A})=0$. Using Theorem~\ref{thm:tilting}, we are done.
\end{proof}

Recall that an algebra $A$ is \emph{minimal representation-infinite} if it is of infinite representation type and any proper quotient algebra is of finite representation type; see \cite{Rin11}. In the following result, we mention that (2) is due to \cite{Zhang, LyW}; compare \cite{Chang}.

\begin{prop}\label{prop:conj}
    The conjecture  {\rm  (S$\tau$C) } holds for the following classes of algebras.
    \begin{enumerate}
    \item Local algebras.
        \item Syzygy-finite  Gorenstein algebras, including algebras with finite global dimension.
        \item Algebras of finite representation type.
        \item Minimal representation-infinite algebras.
    \end{enumerate}
\end{prop}

\begin{proof}

 Let $T$ be a self-orthogonal $\tau$-tilting $A$-module. We may assume that $T$ is not faithful.

For (1), we refer to Remark~\ref{rem:local}.   For a Gorenstein algebra, it is syzygy-finite if and only if it is CM-finite. By \cite[Corollary~2.3]{Mar}, any self-orthogonal module over a CM-finite Gorenstein algebra has finite projective dimension.  Then Proposition~\ref{prop:dell} applies to (2).

  We observe that in cases (3) and (4), the exact category ${\rm fac}(T)$ is syzygy-finite, and thus has finite global delooping level; see Example~\ref{exm:dell}.   For (4), we observe that ${\rm fac}(T)$ is a subcategory of $\bar{A}\mbox{-mod}$.  Therefore, the quotient algebra $\bar{A}$ is of finite representation type. Consequently, the category  ${\rm fac}(T)$  is syzygy-finite.
\end{proof}

The following theorem indicates that the delooping level might play a role in the study of  (S$\tau$C).

\begin{thm}\label{thm:dell-B}
    Let $T$ be a self-orthogonal $\tau$-tilting $A$-module. Set $B={\rm End}_A(T)^{\rm op}$. If ${\rm gl.dell}(B)$ is finite, then $T$ is tilting.
\end{thm}

Here, $B={\rm End}_A(T)^{\rm op}$ is the opposite algebra of the endomorphism algebra of $T$. In particular, $T$ becomes an $A$-$B$-bimodule, and $DT$ a $B$-$A$-bimodule.

\begin{proof}
By Proposition \ref{prop:dell}, it suffices to show ${\rm dell}_T(D\bar{A})<\infty$.

Set $\bar{A}=A/{{\rm Ann}(T)}$. Recall that $T$ is a tilting $\bar{A}$-module. By the classical theorem of Brenner-Butler \cite[VI.3]{ASS}, we have an equivalence between exact categories
$$F={\rm Hom}_A(T, -)\colon {\rm fac}(T) \stackrel{\sim}\longrightarrow {\rm sub}(DT)=\mathcal{S}.$$
Since $F(D\bar{A})\simeq D(T)$, we have
$${\rm dell}_T(D\bar{A})={\rm dell}_{\mathcal{S}}(DT).$$
It remains to prove ${\rm dell}_\mathcal{S}(DT)<\infty$.

Assume that ${\rm gl.dell}(B)=g<\infty$. By ${\rm dell}_B (DT)\leq g$, we have a $B$-module $Y$ such that $\Omega^g_B (DT)$ is a direct summand of $\Omega^{g+1}_B (Y)$ in $B$-$\underline{\rm mod}$. Similarly, by ${\rm dell}_B (Y)\leq g$, we have  a $B$-module $Z$ such that $\Omega^g_B (Y)$ is a direct summand of $\Omega^{g+1}_B (Z)$ in $B$-$\underline{\rm mod}$. Combining these facts, we infer that $\Omega^{g}_B(DT)$ is a direct summand of $\Omega_B^{g+2}(Z)$.

Since $\mathcal{S}$ is a torsionfree class in $B\mbox{-}{\rm mod}$ containing all projective $B$-modules, we observe that $N=\Omega_B (Z)$ belongs to $\mathcal{S}$. Moreover, $\Omega_\mathcal{S}$ coincides with the restriction of $\Omega_B$ on $\underline{\mathcal{S}}$. It follows that $\Omega^{g}_\mathcal{S}(DT)$ is a direct summand of $\Omega^{g+1}_\mathcal{S}(N)$ in $\underline{\mathcal{S}}$. By the very definition of ${\rm dell}_\mathcal{S}(DT)$, we conclude that   ${\rm dell}_\mathcal{S}(DT)\leq g <\infty$.
\end{proof}

We mention the work \cite{KR}, where  explicit modules with infinite delooping level are studied.

\vskip 10pt

\noindent{\bf Acknowledgements}. The authors thank the referee for many helpful suggestions that improve the paper, and especially for pointing out an error in the original proof of Theorem~\ref{thm:dell-B}.
The authors  thank Wen Chang for sharing  \cite{CT}, which provides a counterexample to both (SWC) and (SFC). The example, however, is not a counterexample to (S$\tau$C).  This project is supported by National Key R$\&$D Program of China (No. 2024YFA1013801) and  the National Natural Science Foundation of China (Nos. 12325101, 12131015, 12371015, 12371038, and 12571034).

\bibliography{}

\vskip 10pt

 {\footnotesize \noindent  Xiao-Wu Chen\\
 School of Mathematical Sciences, University of Science and Technology of China, Hefei 230026, Anhui, PR China\\

 \footnotesize \noindent Zhi-Wei Li, Xiaojin Zhang\\
School of Mathematics and Statistics, Jiangsu Normal University, Xuzhou 221116, Jiangsu, PR China\\

 \footnotesize \noindent Zhibing Zhao\\
 School of Mathematical Sciences, Anhui University,  Hefei 230601, Anhui, PR China}

\end{document}